\documentclass[11pt]{article}
\usepackage{amsthm, amsmath, amssymb, amsfonts, url, booktabs, tikz, setspace, fancyhdr, amsbsy}
\usepackage[margin = 1in]{geometry}
\usepackage{hyperref, enumerate}
\usepackage{esint}
\usepackage{verbatim}
\usepackage{subfig}

\newtheorem{theorem}{Theorem}[section]
\newtheorem{proposition}[theorem]{Proposition}
\newtheorem{lemma}[theorem]{Lemma}

\theoremstyle{definition}
\newtheorem{definition}[theorem]{Definition}
\newtheorem{remark}[theorem]{Remark}

\DeclareMathOperator*{\argmin}{arg\,min}

\newcommand{\norm}[1]{\left\lVert#1\right\rVert}

\newcommand{\R}{\mathbb{R}}
\newcommand{\N}{\mathbb{N}}

\newcommand{\dx}{\,\mathrm{d}x}
\newcommand{\dom}{\,\mathrm{d}\omega}
\newcommand{\dy}{\,\mathrm{d}y}

\newcommand{\NN}{\mathcal{N}}
\newcommand{\LL}{\mathcal{L}}
\newcommand{\om}{\omega}
\numberwithin{equation}{section}

\def\ocirc#1{\ifmmode\setbox0=\hbox{$#1$}\dimen0=\ht0
    \advance\dimen0 by1pt\rlap{\hbox to\wd0{\hss\raise\dimen0
    \hbox{\hskip.2em$\scriptscriptstyle\circ$}\hss}}#1\else
    {\accent"17 #1}\fi}

\begin{document}

\title{Convergence analysis of unsupervised Legendre-Galerkin neural networks for linear second-order elliptic PDEs}

\author{
Seungchan Ko\thanks{Department of Mathematics, Sungkyunkwan University, Suwon, Republic of Korea. Email: \tt{ksm0385@skku.edu}},
~Seok-Bae Yun\thanks{Department of Mathematics, Sungkyunkwan University, Suwon, Republic of Korea. Email: \tt{sbyun01@skku.edu}},
~and
~Youngjoon Hong\thanks{Department of Mathematics, Sungkyunkwan University, Suwon, Republic of Korea. Email: \tt{hongyj@g.skku.edu}}
}
\date{\today}

\maketitle

~\vspace{-1.5cm}

\begin{abstract}
In this paper, we perform the convergence analysis of unsupervised Legendre--Galerkin neural networks (ULGNet), a deep-learning-based numerical method for solving partial differential equations (PDEs). 
Unlike existing deep learning-based numerical methods for PDEs, the ULGNet expresses the solution as a spectral expansion with respect to the Legendre basis and predicts the coefficients with deep neural networks by solving a variational residual minimization problem. Since the corresponding loss function is equivalent to the residual induced by the linear algebraic system depending on the choice of basis functions, we prove that the minimizer of the discrete loss function converges to the weak solution of the PDEs. Numerical evidence will also be provided to support the theoretical result. Key technical tools include the variant of the universal approximation theorem for bounded neural networks, the analysis of the stiffness and mass matrices, and the uniform law of large numbers in terms of the Rademacher complexity.
\end{abstract}


\smallskip


\section{Introduction}
In recent years, modern machine learning techniques using deep neural networks have had tremendous success in various areas from science and engineering: computer vision \cite{comv}, natural language processes \cite{DLN}, time series analysis \cite{TSA}, cognitive science \cite{CSS}, and so on. 
The notable advance of deep learning algorithms has offered interesting possibilities for computational innovations. 
In particular, based on the rich expressiveness of neural networks, it has recently begun to gain more attention to solving partial differential equations (PDEs) using a deep learning algorithm.

PDEs are fundamental mathematical models in studying a wide range of phenomena naturally arising in science and engineering. Numerical approximation of PDEs has a long and rich history, and the mathematical theory of the numerical methods is well-established; see, for example, \cite{BSFEM, FEM_Intro} for the finite element methods and \cite{spectral_method} where the spectral method was studied. On the other hand, in order to tackle the problems of high-dimensional PDEs, various deep learning-based techniques solving PDEs were developed such as Deep Ritz Method (DRM) \cite{DRM}, Deep Galerkin Methods (DGM) \cite{DGM} and Physics-Informed Neural Network (PINN) \cite{PINN, VPINN}. 
By utilizing the approximation properties of deep neural networks and statistical learning theory to numerical PDEs, these new approaches opened a new area
of research which is recently called scientific machine learning. 
Since neural networks have a great capability of approximating nonlinear functions, approximated solutions with parameters which are updated through machine learning-based optimization techniques have shown notable performance in solving PDEs numerically. 
Those methods introduced certain types of residuals to define a loss function induced from the given PDEs, which is minimized by an optimization algorithm over the randomly-selected training points on the domain. 
In particular, the PINN exploits randomly-chosen collocation points as training data in the space-time domain, and hence, the PINN is available for time-dependent high-dimensional PDEs, on the computational domains of complex geometry \cite{PINN_1, PINN_2, PINN_3, PINN_4, PINN_5, PINN_6}.

However, these approaches have some drawbacks. For example, only a single instance from a fixed input data, including initial conditions, boundary conditions, and external force terms is available at a time. This means that if the input data is changed, the whole training process needs to be repeated. 
Hence, real-time predictions for varying input data are not possible for the aforementioned methods. 
Some techniques have been developed to overcome this problem; data-supervised or data-driven methods (DDM) \cite{CHKP2020, multi_1, multi_2}, the Fourier neural operator (FNO) \cite{FNO}, and the DeepONet (DON) \cite{DON}. 
However, training these models requires a database, which should be provided in advance by solving the equations analytically or numerically. To generate a reliable training dataset, a large number of numerical solutions need to be provided by massive numerical computations, which is inefficient and often takes a very long time.
Furthermore, using the class of neural networks as the solution space makes it difficult to impose boundary values. In general, for a function represented by neural networks, it is not straightforward how to incorporate certain values on the boundary of the given domain. One common strategy to deal with this issue is to add penalty terms to loss functions. However, it is known that the boundary penalty term might cause worse convergence of the approximate solution and some technical difficulties may arise.

To address the aforementioned issues, various attempts have been made in recent years. In particular, in \cite{CKH2022} the unsupervised Legendre--Galerkin neural network (ULGNet) was proposed to solve various PDEs without a training dataset. 
This approach is based on the classical numerical method, the Legendre--Galerkin spectral element method for the boundary value problem \cite{spectral_2, spectral_method}. 
By choosing the Legendre polynomials as basis functions, the corresponding polynomial coefficients are generated via deep neural networks whose input is data of PDEs such as external forcing terms or initial conditions.
Then the predicted solution of the ULGNet can be obtained by a linear combination of the coefficients and the Legendre basis functions. 
For the training, the loss function is defined as a residual of summation of weak formulations which take their test function as Legendre basis functions. 
Then the random points are sampled from the parametric domain which determines the source term (input of the neural network) and the neural network is trained in the direction of minimizing the loss function. 
Consequently, the ULGNet can overcome the shortcomings of various models in scientific machine learning. 
Since it is based on unsupervised learning process, the expensive training database is not required while the algorithm is able to predict multiple instances of PDEs.
In addition, the predicted solution can have the exact boundary values as given in accordance with the choice of spatial basis functions. 
Using a weak form as the loss function instead of the strong form used in other literature (e.g. PINN \cite{PINN} or DGM \cite{DGM}) is advantageous in some aspects, since we can avoid
high-order numerical derivatives that cause large numerical errors, which may be able to enhance the performance of neural networks. 
Moreover, using the Legendre basis functions can also improve the accuracy of the approximation because the numerical integration that appeared in the ULGNet algorithm can be computed by the Gauss–Legendre quadrature rule which is highly accurate. 
From the numerical experiments performed in \cite{CKH2022}, it was shown that the ULGNet can compute the approximate solutions with high accuracy in various situations. 
Especially, this technique performs very well for the singularly-perturbed and boundary layer problems, which is mainly due to the fact that we have freedom in manipulating the basis functions, and hence we are able to design an enriched scheme in this framework. 
The more detailed discussion of ULGNet will be studied in Section \ref{prelim}. 

Despite a wide range of successful applications for neural network-based PDE solvers, the rigorous analysis of such methods is not well developed and only a few theoretical results are currently known. 

Some papers provided evidence of convergence for the neural network-based method \cite{DGM, misc_1}. Note, however, that these papers only proved that there exists a sequence of neural networks which converges to the true solution, rather than showing the actual convergence of approximate solutions. Some other previous research papers on the analysis of neural network-based algorithms solving PDEs investigated on the representation aspect. 
In other words, they focused on whether a solution of the given PDE is well-approximated by neural networks with quantitative control of numerical errors; see, for example, \cite{approx_1, approx_2, boundp, rade_upper_3}. 
On the other hand, if we fix an ansatz space, the generalization error can be estimated by analyzing complexities using covering numbers \cite{gen_err_1, rade_upper_3}. 
More recently, the analysis for the DRM in a high-dimensional space has been established; see \cite{rade_upper_3} where the class of two-layer neural networks in the spectral Barron framework was studied and a more general setting was considered in \cite{bdd_para}. Moreover, some recent papers investigated the generalization error analysis of the PINN \cite{PINN_anal_1, PINN_anal_2, PINN_anal_3} based on the residual minimization problem \cite{res_min_1, SZK2020}. 
More precisely, \cite{PINN_anal_1} focused on the consistency of the loss function so that the approximate solution converges to the exact solution as the number of training samples increases with the assumption of vanishing optimization error. The authors in \cite{PINN_anal_2} conducted an a-posteriori error analysis and showed that the generalization error can be bounded by the training error and the quadrature error under some assumptions of stability of the given PDEs. 
The work in \cite{PINN_anal_3} proved a priori generalization error estimates for a class of second-order linear PDEs by assuming that the true solutions of PDEs are contained in a Barron-type space, and \cite{SZK2020} obtained both a priori and a posterior error estimates for residual minimization problems in Sobolev spaces setting. 
More results on the theoretical analysis of neural network-based methods of PDEs can be found in \cite{misc_1, misc_2, misc_3, rade_upper_4}.

In this perspective, the central aim of the present paper is to perform the theoretical analysis on the convergence of ULGNet, which supports the experimental results in \cite{CKH2022}. Note that all the aforementioned convergence results were about the neural network for the single instance prediction. To the best of our knowledge, this is the first paper which studies convergence analysis for the unsupervised neural network for the prediction of multiple instances.
More precisely, we consider the following second-order, self-adjoint linear elliptic PDE: for some $\nu\geq0$,
\begin{equation}\label{intro_eq_main}
    -\Delta u+\nu u = f\quad{\rm{in}}\,\,I^d=[-1,1]^d,
\end{equation}
with either Dirichlet or Neumann boundary conditions. As mentioned earlier, in the ULGNet setting, the external forcing term $f$ is parametrized by random samples $\om$ from a parametric domain which we shall denote by $\Omega$, and hence the solution $u$ also depends on the parametric variable $\om\in\Omega$. Therefore, a proper distance to be measured is the Bochner-type norm where the distance between the true solution and the approximate solution is integrated over both $I^d$ and $\Omega$. The mathematical setting and detailed problem formulation will be discussed in a later section. 

In general, for the neural network-based methods, the theoretical analysis heavily depends on the form of given PDEs and the type of boundary conditions. For example, the authors of \cite{rade_upper_3} established the convergence of the DRM of the equation \eqref{intro_eq_main} for both $\nu=0$ and $\nu\neq0$, and the analysis of the case $\nu\neq0$ was much more complicated than the one for the case $\nu=0$. Also, by referring \cite{rade_upper_3, bdd_para}, we note that the analysis with the Dirichlet boundary condition is quite different from the analysis for the Neumann condition. As it will be made clear in the later analysis, we will develop a unified approach that can cover various settings with similar analysis, such as the analysis for general $\nu\geq0$ and either Dirichlet or Neumann boundary conditions.

The rest of the present paper is organized as follows. In Section \ref{prelim} we will introduce some preliminaries that will be used throughout the paper, including notations, a detailed description of the ULGNet, our ansatz class, and the problem setting under consideration. In Section \ref{aca} we shall perform the convergence analysis of the residual minimization problem under the abstract setting and this analysis will be applied to concrete PDE problems in Section \ref{PDE_app}. In Section \ref{num_exp}, we provide some numerical experiments which confirm the theoretical result and a concluding remark will be discussed in \ref{con_rmk}.

\section{Preliminaries and mathematical setting}\label{prelim}
In this section, we introduce some preliminaries related to the numerical approximation of the given problem \eqref{main_eq}, which will be used throughout the paper. 
For $1\leq p\leq\infty$ and $k\in\N$, we denote by $L^p(\Omega)$ the Lebesgue space and by $W^{k,p}(\Omega)$ the standard Sobolev space, with the abbreviation $H^k(\Omega)=W^{k,2}(\Omega)$, and we write as $H^1_0(\Omega)$ the subspace of $H^1(\Omega)$ consisting of the functions with zero trace. For a real Banach space $X$, we shall denote by $C(\Omega;X)$ the subspace of the Bochner space $L^{\infty}(\Omega;X)$, which consists of the continuous functions from $\Omega$ to $X$. Moreover, $(\Omega,\mathcal{G},\mathbb{P})$ signifies a probability space, where $\Omega$ is the sample space consisting of all possible outcomes, $\mathcal{G}\subset2^{\Omega}$ is the $\sigma$-algebra, and $\mathbb{P}$ is a probability measure.

For the sake of simplicity, we write $\|\cdot\|_p=\|\cdot\|_{L^p(\Omega)}$ and $\|\cdot\|_{k,p}=\|\cdot\|_{W^{k,p}(\Omega)}$, and furthermore, $|\cdot|$ denotes the Euclidean norm. Throughout the paper, for two vectors $\boldsymbol{a}$ and $\boldsymbol{b}$, $\boldsymbol{a}\cdot \boldsymbol{b}$ means their scalar product. 
Also, for any Lebesgue measurable set $Q\subset\R^d$, $|Q|$ stands for the standard Lebesgue measure of $Q$, and $C$ denotes a generic positive constant, which may differ at each appearance. 
Finally, for two different quantities $A$ and $B$, the notation $A\lesssim B$ means that there exists a positive constant $C>0$ such that $A\leq CB$. 

In order to simplify the presentation, we shall restrict ourselves to the following one-dimensional problem: for some $\nu\geq0$,
\begin{equation}\label{main_eq}
    -u''+\nu u=f\quad{\rm{in}}\,\,I=(-1,1),
\end{equation}
with either the homogeneous Dirichlet boundary condition $u(-1)=u(1)=0$ or the Neumann boundary condition $u'(-1)=u'(1)=0$. 
As will be discussed in Section \ref{sec:multi_dim}, it can be extended to the multi-dimensional problem with general boundary conditions.  

\subsection{Legendre--Galerkin neural network}\label{ULGNet_intro}
In this section, we shall introduce one of the spectral element methods for solving a two-point boundary value problem, the so-called {\textit{Legendre--Galerkin method}}. 
Based on this numerical method, we describe the unsupervised Legendre--Galerkin neural networks (ULGNet) proposed in \cite{CKH2022}, which we aim to analyze in the present paper. 

To begin, we discuss some basic properties of Legendre polynomials that play a role in basis functions for our numerical approximation. 
Legendre polynomials, $L_n(x)$, are mutually orthogonal polynomials with respect to $L^2$-inner product, defined on the interval $[-1,1]$. 
Here we collect some relevant results for the Legendre polynomials (see, for example, \cite{spectral_method}):
\begin{itemize}
    \item Three-terms recurrence relation:
    \[(n+1)L_{n+1}(x)=(2n+1)xL_n(x)-nL_{n-1}(x),\quad {\rm{with}}\,\,L_0(x)=1\,\,{\rm{and}}\,\,L_1(x)=x.\]
    \item The $n$-th Legendre polynomial can be written as the expansion
    \[L_n(x)=\frac{1}{2^n}\sum^{[n/2]}_{k=0}(-1)^k\frac{(2n-2k)!}{2^nk!(n-k)!(n-2k)!}x^{n-2k}.\]
    \item Sturm--Liouville problem:
    \[\left((1-x^2)L'_n(x)\right)'+n(n+1)L_n(x)=0.\]
    \item Rodriques' formula:
    \[L_n(x)=\frac{1}{2^nn!}\frac{{\rm{d}}^n}{{\rm{d}}x^n}\left[\left(x^2-1\right)^n\right].\]
\end{itemize}
To make the stiffness and mass matrices to be sparse in the numerical approximation, and enable us to impose exact boundary conditions, we adopt the following compact combination of Legendre Polynomials as basis functions:
\begin{equation}\label{basis}
    \phi_k(x)=L_k(x)+a_kL_{k+1}(x)+b_kL_{k+2}(x),
\end{equation}
where the parameters $\{a_k,b_k\}$ are chosen to satisfy the boundary condition of the given problem. In fact, it is straightforward to show that for all $k\in\mathbb{N}$, the set $\{a_k,b_k\}$ can be uniquely determined so that the basis function \eqref{basis} satisfies the given boundary condition either in Dirichlet sense or in Neumann sense. 

Let us discuss the Legendre--Galerkin spectral method. 
The variational form of the equation \eqref{main_eq} can be written as
\begin{equation}\label{WF}
    \int_I\left(u'(x)v'(x)+\nu u(x)v(x)\right)\dx=\int_If(x)v(x)\dx\quad\forall v\in H,
\end{equation}
for a suitable Hilbert space $H$. 
The existence and uniqueness of the solution  \eqref{WF} follow by the Lax--Milgram theorem \cite{evans, GT}. The idea of the  Legendre--Galerkin method is to approximate the solution $u\in H$ by
\begin{equation}\label{sol_expan}
    u_N(x)=\sum^{N-1}_{k=1}\alpha_k\phi_k(x),
\end{equation}
where $N$ is the number of nodal points on $I$ and $\phi_k$ is defined in \eqref{basis}. The coefficients $\{\alpha_k\}_{k=1}^{N-1}$ can be determined by solving the discrete approximation of \eqref{WF}, namely,
\begin{equation}\label{DWF}
    \int_I\left(u'_N(x)v'_N(x)+\nu u_N(x)v_N(x)\right)\dx=\int_If(x)v_N(x)\dx\quad\forall v_N\in H_N,
\end{equation}
where $H_N={\rm{span}}\,\{\phi_1,\phi_1,\cdots,\phi_{N-1}\}$. More precisely, if we write $S=(S_{ij})$, $M=(M_{ij})\in\R^{(N-1)\times(N-1)}$, $F=(F_j)\in\R^{N-1}$ with
\begin{equation}\label{vec_form}
    S_{ij}=\int_I\phi'_i(x)\phi'_j(x)\dx,\quad
    M_{ij}=\int_I\phi_i(x)\phi_j(x)\dx,\quad
    F_j=\int_If(x)\phi_j(x)\dx,
\end{equation}
\eqref{DWF} is equivalent to the linear algebraic system
\begin{equation}\label{LAS}
    (S+\nu M)\alpha=F,
\end{equation}
for $\alpha=(\alpha_j)\in\R^{N-1}$, and \eqref{DWF} can be exactly solved by
\begin{equation}\label{solve_coef}
    \alpha=(S+\nu M)^{-1}F.
\end{equation}

\begin{figure}
\includegraphics[width=0.9\textwidth ]{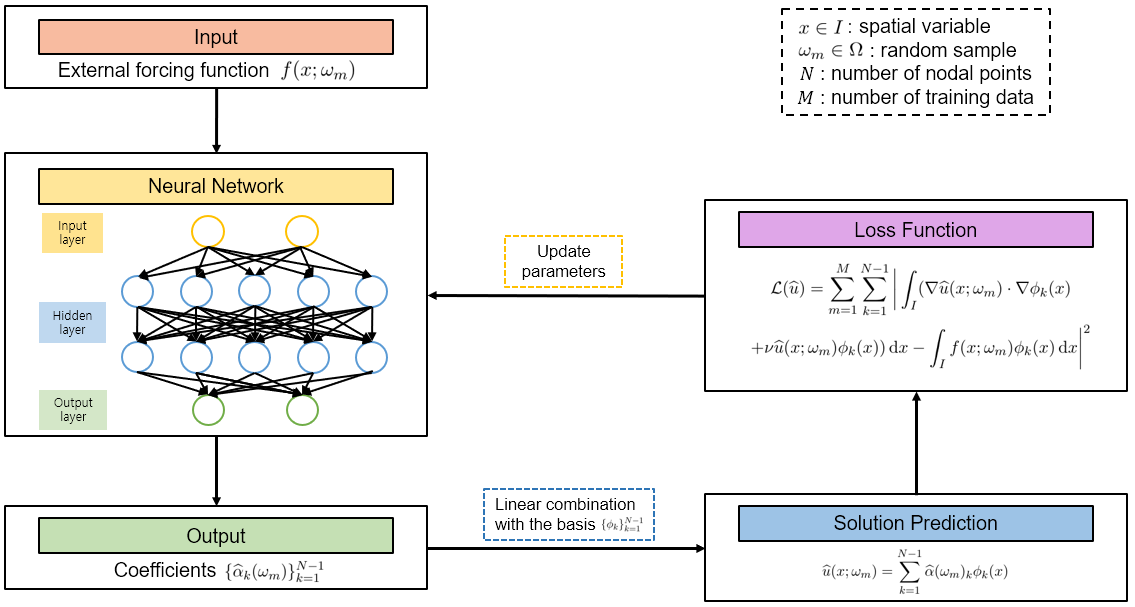}
\caption{Schematic description of the unsupervised Legendre--Galerkin neural network (ULGNet).\label{overall_pic}}
\end{figure}

The Unsupervised Legendre--Galerkin neural network, proposed in \cite{CKH2022}, is based on the above Legendre--Galerkin spectral element methods. For each given $f$, instead of computing the coefficients by using \eqref{LAS}-\eqref{solve_coef}, we approximate the coefficients $\alpha$ by a deep neural network. The input of the neural network is external forcing terms, parametrized by the random parameter $\om$ contained in the (possibly high-dimensional) probability space $(\Omega,\mathcal{G},\mathbb{P})$, where $\Omega$ is assumed to be compact. One typical example is the random external force given by
\[f(x;\om)=\om_1\sin(2\pi\om_2x)+\om_3\cos(2\pi\om_4x),\]
where $\om=(\om_1,\om_2,\om_3,\om_4)$ is i.i.d. uniformly distributed with $\omega_j \in [a,b]$ for $a, b \in \mathbb{R}^+$ and $1 \leq j \leq 4$. Throughout the whole paper, we will interpret $f(x;\om)$ as a bivariate function defined on $I\times\Omega$, and assume that
\begin{equation}\label{f_ass}
    f(x;\om)\in C(\Omega;L^1(I)).
\end{equation}
This input feature passes through the neural network and generates the coefficients $\{\widehat{\alpha}_k\}$ as an output. Then we reconstruct the solution by
\begin{equation}\label{sol_recon}
    \widehat{u}(x;\om)=\sum_{k=1}^{N-1}\widehat{\alpha}_k(\om)\phi_k(x).
\end{equation}
In order to train the neural network, we utilize the variational formulation \eqref{DWF} for the residual minimization and define the population loss function by
\begin{equation}\label{pop_loss}
    \LL(\alpha)=\mathbb{E}_{\om\sim\mathbb{P}_{\Omega}}\left[\sum^{N-1}_{k=1}\bigg|\int_I\left(\widehat{u}'(x;\om)\phi'_k(x)+\nu \widehat{u}(x;\om)\phi_k(x)\right)\dx-\int_If(x;\om)\phi_k(x)\dx\bigg|^2\right].
\end{equation}
In practice, when we compute the approximate solution, we minimize the empirical loss function which is the Monte--Carlo integration of \eqref{pop_loss}:
\begin{equation}\label{emp_loss}
    \LL^M(\alpha)=\frac{|\Omega|}{M}\sum^M_{m=1}\sum^{N-1}_{k=1}\bigg|\int_I\left(\widehat{u}'(x;\om_m)\phi'_k(x)+\nu \widehat{u}(x;\om_m)\phi_k(x)\right)\dx-\int_If(x;\om_m)\phi_k(x)\dx\bigg|^2,
\end{equation}
where $M\in\mathbb{N}$ is a number of training data and $\{\om_m\}_{m=1}^{M}$ is an i.i.d. sequence of random variables distributed according to $\mathbb{P}_{\Omega}$. For each epoch, once the parameters of the neural network are trained in the direction of minimizing the empirical loss function $\LL^M$, the given external force functions pass through this updated neural network to produce more refined coefficients, and this process is repeated until the sufficiently small loss is achieved. Thanks to the intrinsic structure described above, ULGNet can be trained without training data and incorporate the exact boundary values to the approximate solutions. A schematic diagram of the ULGNet algorithm is depicted in Figure \ref{overall_pic}, and more details with some numerical experiments can be found in \cite{CKH2022}.

\subsection{Feed-forward neural networks}\label{neural_network}
In this section, we define a class of neural networks, which will serve as an ansatz space throughout the whole paper. For a positive integer $L\in\mathbb{N}$, we define an $L$-layer feed-forward neural network by a function $f^L(x):\R^{n_0}\rightarrow\R^{n_L}$, which is defined recursively by
\begin{equation}\label{nn_def}
    f^{\ell}(x)=W^{\ell}\sigma(f^{\ell-1}(x))+b^{\ell}\,\,\,{\rm{for}}\,\,2\leq\ell\leq L\quad{\rm{and}}\quad f^1(x)=W^1x+b^1.
\end{equation}
Here $W^{\ell}\in\R^{n_{\ell}\times n_{\ell-1}}$ is the weight matrix and $b^{\ell}\in\R^{n_{\ell}}$ is the bias vector for the $\ell$-th layer. Also, $\sigma:\R\rightarrow\R$ denotes the activation function, and for a vector $x=(x_1,\cdots,x_k)$, $\sigma(x)$ means the vector $(\sigma(x_1),\cdots,\sigma(x_k))$. We shall write the architecture of a network by the vector ${\vec{\boldsymbol{n}}}=(n_0,\cdots,n_L)$, the family of network parameters by $\theta:=\theta_{\vec{\boldsymbol{n}}}=\{(W^1,b^1),\cdots,(W^L,b^L)\}$, and its realization as a neural network by $\mathcal{R}[\theta](x)$. For a given network architecture $\vec{\boldsymbol{n}}$, the collection of all possible parameters is denoted by
\begin{equation}\label{nn_para_set}
    \Theta_{\vec{\boldsymbol{n}}}=\left\{\{(W^{\ell},b^{\ell})\}^L_{\ell=1}:W^{\ell}\in\R^{n_{\ell}\times n_{\ell-1}},\,b^{\ell}\in\R^{n_{\ell}}\right\}.
\end{equation}
 For the index for neural networks, we will follow the notations from \cite{SZK2020, notation}. For two different neural networks $\theta_i$, $i=1,2$, whose architectures are denoted by $\vec{\boldsymbol{n}}_i=\left(n^{(i)}_0,\cdots,n^{(i)}_{L_i}\right)$, we write $\vec{\boldsymbol{n}}_1\subset\vec{\boldsymbol{n}}_2$ provided that for all $\theta_1\in\Theta_{\vec{\boldsymbol{n}}_1}$, there exists $\theta_2\in\Theta_{\vec{\boldsymbol{n}}_2}$ satisfying $\mathcal{R}[\theta_1](x)=\mathcal{R}[\theta_2](x)$ for any $x\in\R^{n_0}$.
 
 Now let $\{\vec{\boldsymbol{n}}_n\}_{n\geq1}$ be a sequence of network architectures with the inclusion $\vec{\boldsymbol{n}}_n\subset \vec{\boldsymbol{n}}_{n+1}$ for any $n\in\mathbb{N}$. We shall define the corresponding sequence of neural network families by
\begin{equation}\label{nn_class}
    \NN^*_n=\{\mathcal{R}[\theta]:\theta\in\Theta_{\vec{\boldsymbol{n}}_n}\}.
\end{equation}
By the definition, it is easy to verify that $\NN^*_n\subset\NN^*_{n+1}$ for all $n\in\mathbb{N}$. We then state the following theorem on the universal approximation property.
\begin{theorem}\label{ass_1}
Let $K$ be a compact subset of $\R^{d}$ and $g\in C(K,\R^{D})$. Then we have
\begin{equation}\label{UAT}
    \lim_{n\rightarrow\infty}\inf_{\hat{g}\in\NN^*_n}\|\hat{g}-g\|_{C(K)}=0.
\end{equation}
\end{theorem}
It can be proved that this statement holds under various scenarios. For example, the classical form of the universal approximation theorem for arbitrary width and bounded depth coincides with the above theorem provided that the activation function is not polynomial \cite{UA_1, UA_2, UA_3}. In this case, we only consider the class of two-layer neural networks ($L=2$), and $n$ can be interpreted as the index related to the number of neurons in the hidden layer. On the other hand, in the recent 
work of \cite{terry_lion}, the authors discussed the dual scenario: the networks of bounded width and arbitrary depth. The authors assumed that the activation function is a non-affine continuous function that is continuously differentiable at at least one point, with a nonzero derivative at that point. Under this assumption, they considered the space of feed-forward neural networks with an arbitrary number of hidden layers each with $d+D+2$ neurons, and proved \eqref{UAT}. In this case, $n_i=d+D+2$ for all $i\in\{1,\cdots,L-1\}$ and $n$ can be regarded as the index concerning the number of hidden layers ($n=L$).

In this paper, we will consider the bounded and continuous target functions. More precisely, in Section \ref{setting}, the function we aim to approximate by neural networks is $\alpha^*(\om)$ which is the solution of \eqref{LAS}, where $S$, $M$ are defined as before, but each component of the load vector $F_j(\om)=\int_If(x;\om)\phi_j(x)\dx$ is the function of $\om\in\Omega$. By the assumption \eqref{f_ass}, we know that $\alpha^*(\om)$ is continuous and bounded. Therefore, it suffices to consider the family of bounded neural networks as our ansatz space. 

Let us assume that the target function $\alpha^*$ in \eqref{solve_coef} has the bound $\sup_{\om\in\Omega}|\alpha^*(\om)|<C_{\alpha}$, and let $\sigma^*$ be a bounded activation function with continuous inverse (e.g. sigmoid or hyperbolic tangent activation functions) and $\sigma_{-}:=\inf_{x\in\R}\sigma^*(x)$ and $\sigma_{+}:=\sup_{x\in\R}\sigma^*(x)$. Next, we define the affine linear function $h:[\sigma_{-},\sigma_{+}]\rightarrow[-C_{\alpha},C_{\alpha}]$ with $h(\sigma_{-})=-C_{\alpha}$, $h(\sigma_{+})=C_{\alpha}$. Note that the function $h$ also has a continuous inverse.

We then consider the following class of neural networks:
\begin{equation}\label{bounded_nn_class}
    \NN_n=\left\{h\circ\sigma^*(g_n):g_n\in\NN_n^*\right\}.
\end{equation}
Then it is obvious that
\begin{equation}\label{bddness_NN}
    \|g\|_{C(\Omega)}\leq C_{\alpha},\quad\forall g\in\bigcup_{n\in\mathbb{N}}\NN_n.   
\end{equation}
By the definitions of $h$ and $\sigma^*$, $(\sigma^*)^{-1}\circ h^{-1}(\alpha^*)$ is well-defined, and is continuous on the compact set $\Omega$.
Then by Theorem \ref{ass_1}, we have
\begin{equation}\label{UAT_var_1}
    \lim_{n\rightarrow\infty}\inf_{\hat{g}\in\NN^*_n}\bigg\|\hat{g}-(\sigma^*)^{-1}\circ h^{-1}(\alpha^*)\bigg\|_{C(\Omega)}=0.
\end{equation}
Then from \eqref{UAT_var_1}, we obtain by the continuity of the function $h\circ\sigma^*(\cdot)$ that
\[
\lim_{n\rightarrow\infty}\inf_{\hat{g}\in\NN^*_n}\bigg\|h\circ \sigma^*(\hat{g})-\alpha^*\bigg\|_{C(\Omega)}=0.
\]
Therefore, by the definitions of $\NN_n$ and $\NN^*_n$, we have that
\begin{equation}\label{UAT_var_2}
    \lim_{n\rightarrow\infty}\inf_{\tilde{g}\in\NN_n}\|\tilde{g}-\alpha^*\|_{C(\Omega)}=
    \lim_{n\rightarrow\infty}\inf_{\hat{g}\in\NN^*_n}\bigg\|h\circ \sigma^*(\hat{g})-\alpha^*\bigg\|_{C(\Omega)}=0.
\end{equation}
The above argument is summarized in the following theorem concerning the universal approximation properties for bounded neural networks.
\begin{theorem}\label{UAT_main_thm}
Assume that $\Omega$ is compact and $\alpha^*:\Omega\rightarrow\R^{N-1}$ is a continuous function defined in \eqref{solve_coef}. Then we have
\begin{equation}\label{UAT_var_3}
    \lim_{n\rightarrow\infty}\inf_{\tilde{g}\in\NN_n}\|\tilde{g}-\alpha^*\|_{C(\Omega)}=0.
\end{equation}
\end{theorem}

\subsection{Problem formulation}\label{setting}
In this section, we shall state the problem under consideration in detail, and describe the mathematical setting and the strategy for the analysis. Henceforth, for given external force $f(x;\om)$ parametrized by $\om\in\Omega$, $u(x;\om)$ denotes the corresponding true solution (weak solution) of the equation \eqref{main_eq} either with Dirichlet boundary condition or Neumann boundary condition, and $u_N$ signifies the Legendre--Galerkin approximation 
\begin{equation}\label{for_0}
    u_N(x,\om)=\sum_{k=1}^{N-1}\alpha^*_k(\om)\phi_k(x),    
\end{equation}
where $\{\alpha^*_k(\om)\}$ is the set of coefficients computed by \eqref{solve_coef}.
Here we shall use the same notations $S$, $M$ and $F(\om)$ as defined in \eqref{vec_form}.

Note that both $u$ and $u_N$ depend on the external force $f(x;\om)$ and hence, depend on the parametric variable $\om\in\Omega$. In particular, $\alpha^*$ can be regarded as a function contained in $C(\Omega,\R^{N-1})$ provided that $f(x;\om)$ satisfies the assumption \eqref{f_ass} . In this regard, we can interpret $\alpha^*$ as the solution to the minimization problem which is equivalent to \eqref{DWF}:
\begin{equation}\label{for_1}
    \alpha^*=\argmin_{\alpha\in C(\Omega,\R^{N-1})} \LL(\alpha),
\end{equation}
where $\LL$ is the population loss function defined in \eqref{pop_loss}.

Next, we define the neural network approximation of $\alpha^*$. More precisely, we look for $\widehat{\alpha}(n):\Omega\rightarrow\R^{N-1}$ solving the continuous residual minimization problem
\begin{equation}\label{for_2}
    \widehat{\alpha}(n)=\argmin_{\alpha\in\NN_n}\LL(\alpha),
\end{equation}
where the minimizer was found over the family of neural networks $\NN_n$ for $n\in\mathbb{N}$, and we shall write the associate solution by
\begin{equation}\label{for_3}
    u_{N,n}(x;\om)=\sum^{N-1}_{k=1}\widehat{\alpha}(n)_k(\om)\phi_k(x).
\end{equation}

Finally, let us define the solution to the discrete residual minimization problem, namely,
\begin{equation}\label{for_4}
    \widehat{\alpha}(n,M)=\argmin_{\alpha\in\NN_n}\LL^M(\alpha),
\end{equation}
where the empirical loss $\LL^M$ was defined in \eqref{emp_loss}. In addition, we shall write the corresponding solution as
\begin{equation}\label{for_5}
    u_{N,n,M}(x;\om)=\sum^{N-1}_{k=1}\widehat{\alpha}(n,M)_k(\om)\phi_k(x).
\end{equation}

In the present paper, we shall ignore the error caused by the minimization, and assume that we can always find the exact minimizer for the minimization problems \eqref{for_2} and \eqref{for_4}. Therefore, we shall employ $\widehat{\alpha}(n,M)$ as the neural network approximation for the exact coefficient $\alpha^*$, and $u_{N,n,M}$ is our approximate solution computed by the numerical scheme described in Section \ref{ULGNet_intro}.
If we were to summarize once more, $N$ stands for the parameter for the Legendre--Galerkin method (number of basis functions), $n$ is the parameter for the neural network architecture, and $M$ denotes the number of sampling points when we compute the discrete loss function. 

The main objective of this paper is to show the convergence 
\begin{equation}\label{main_error}
    \|u-u_{N,n,M}\|_{L^2(\Omega;H^1(I))}\rightarrow0\quad{\rm{as}}\,\,N,n,M\rightarrow\infty.
\end{equation}
For this purpose, we split the error in \eqref{main_error} into three parts:
\begin{equation}\label{error_split}
    u-u_{N,n,M}=
    (u-u_{N})+
    (u_{N}-u_{N,n})+
    (u_{N,n}-u_{N,n,M}).
\end{equation}
The first term on the right-hand side of \eqref{error_split} is the error for the Legendre--Galerkin approximation. 
It is widely known that the Legendre--Galerkin approximation error is sufficiently small in the numerical analysis, and hence this error can be ignored in this paper.
In fact, it is well known that if the solution $u$ is in $H^m$ for some $m\in\mathbb{N}$, then the corresponding approximate solution $u_N$ is close to $u$ within $\mathcal{O}(N^{-m})$ accuracy. Hence, if the solution $u^*$ is sufficiently smooth, the Legendre--Galerkin method achieves exponential convergence. Furthermore, as $N$ increases, the condition number of the capacitance matrix $S+\nu M$ grows, and hence, the large $N$ may deteriorate numerical convergence in practice. See, for example, Table III in \cite{SM_new} where $N$ increased from $64$ to $128$, but the computed $L^{\infty}$ error increased slightly. 
Therefore, a suitable choice of $N$ guarantees that the approximate error is negligible, which justifies our assumption. In the rest of the paper, we will assume that we have found sufficiently good $N$ and we will fix that $N$ in the analysis.

The second term is called the {\textit{approximation error}}, since it occurs when we approximate the target function with the functions in our ansatz space. The third term is often referred to as the {\textit{generalization error}}, which is the error of the neural network-based approximation on predicting unseen data. 

We will first perform the convergence analysis of both approximation error and generalization error for the corresponding coefficients neural networks, i.e., \begin{equation}\label{nn_error}
    \|\alpha^*-\widehat{\alpha}(n)\|_{L^2(\Omega)} + \|\widehat{\alpha}(n)-\widehat{\alpha}(n,M)\|_{L^2(\Omega)}.
\end{equation}
This error estimate for the neural networks is the crucial step in our analysis. We will discuss the convergence of the error \eqref{nn_error} in Section \ref{aca}, and based on this result, we will conduct the convergence analysis of the approximate solution \eqref{main_error} in Section \ref{PDE_app}.

\section{Abstract convergence analysis}\label{aca}

In this section, we shall perform an abstract convergence analysis for the error \eqref{nn_error}. We begin with an observation on the structure of the loss functions. Note by the definition of  $\LL(\alpha)$ and \eqref{sol_expan} that
\begin{equation}\label{obs_J}
\begin{aligned}
    \LL(\alpha)
    &=\int_{\Omega}\sum^{N-1}_{j=1}\bigg|\sum^{N-1}_{k=1}\alpha_k(\om)\left(\int_I\phi'_k(x)\phi'_j(x)\dx\right)-\int_If(x;\om)\phi_j(x)\bigg|^2\dom\\
    &=\int_{\Omega}\sum^{N-1}_{j=1}\bigg|\left(A\alpha(\om)\right)_j-\left(F(\om)\right)_j\bigg|^2\dom=\norm{A\alpha(\om)-F(\om)}^2_{L^2(\Omega)},
\end{aligned}
\end{equation}
where $A:=S+\nu M$ and $F(\om)$ are defined in \eqref{vec_form}. Similarly, we also have
\begin{equation}\label{obs_JM}
    \LL^M(\alpha)=\frac{|\Omega|}{M}\sum_{m=1}^M|A\alpha(\om_m)-F(\om_m)|^2.
\end{equation}
Here the matter on the loss functions $\LL$ and $\LL^M$ has been converted into the analysis of the residual induced by the matrix equation defined by $A$. Note that the matrix $A$ depends on the structure of the given differential equations, boundary conditions and the choice of basis functions, and hence the general analysis of the matrix $A$ that can cover various PDE settings and numerical methods is important. In the next proposition, we shall present the result on the structure of the matrix $A$ which can cover a wide range of settings of interest.

\begin{proposition}\label{matrix_thm}
Assume that $A\in\R^{d\times d}$ is a symmetric, non-singular matrix and let $\rho_{\min}=\min_i\{|\lambda_i|\}$, $\rho_{\max}=\max_i\{|\lambda_i|\}$ where $\{\lambda_i\}$ is the set of eigenvalues of $A$. Then for any $x\in\R^d$ we have
\begin{equation}\label{eigen_est}
    \rho_{\min}|x|\leq|Ax|\leq\rho_{\max}|x|.
\end{equation}
\end{proposition}
\begin{proof}
By the spectral theorem, we can choose the set $\{x_1,\cdots,x_d\}$, where $x_j$ is the corresponding eigenvector for the eigenvalue $\lambda_j$ of $A$, and $\{x_1,\cdots,x_d\}$ forms an orthonormal basis of $\R^d$. Hence for arbitrary $x\in\R^d$, we can write $x=\sum^d_{j=1}a_jx_j$. for $a_j\in\R$. Then by the orthonormality and the definition of eigenvalue, we have
\[
    |Ax|^2=\bigg|A\sum^d_{j=1}a_jx_j\bigg|^2=\bigg|\sum^d_{j=1}a_jAx_j\bigg|^2=\bigg|\sum^d_{j=1}a_j\lambda_jx_j\bigg|^2=\sum^d_{j=1}(a_j\lambda_j)^2.
\]
Since $A$ is non-singular, all the eigenvalues of $A$ are nonzero, and from the above inequality, we can conclude that
\[
    \rho^2_{\min}|x|^2\leq|Ax|^2\leq\rho^2_{\max}|x|^2,
\]
which completes the proof.
\end{proof}

Throughout the section, we will utilize the matrix $A=S+\nu M$ defined in \eqref{vec_form}, and assume that $A$ is symmetric and non-singular. As it will be made clear in the next section, $A$ is indeed symmetric and non-singular in many cases including the situation under consideration in this paper. Henceforth, we set $\Omega=[a,b]^d$ for $a,b \in \mathbb{R}^+$.

\subsection{Approximation error}
This subsection aims to prove that the approximation error for the neural networks converges to zero, which is encapsulated in the following theorem. The idea of the proof is to make a connection between the error and the loss function evaluated at the minimizer, which actually converges to zero as the network architecture grows.
\begin{theorem}\label{approx_main_thm}
Assume that \eqref{f_ass} holds. For $\alpha^*\in C(\Omega,\R^{N-1})$ defined in \eqref{for_1} and $\widehat{\alpha}(n)\in\NN_n$ defined from \eqref{for_2}, we have that
\[
\|\alpha^*-\widehat{\alpha}(n)\|_{L^2(\Omega)}\rightarrow0\quad{\rm{as}}\,\,n\rightarrow\infty.
\]
\end{theorem}
\begin{proof}
Since $A$ is symmetric and non-singular, by Proposition \ref{matrix_thm}, we have
\begin{equation}\label{mid_1}
    \|\alpha^*-\widehat{\alpha}(n)\|^2_2\lesssim\|A\alpha^*-A\widehat{\alpha}(n)\|^2_2\lesssim\|A\alpha^*-F\|^2_2+\|A\widehat{\alpha}(n)-F\|^2_2=\LL(\widehat{\alpha}(n)),
\end{equation}
where we have used \eqref{LAS}. Since $\widehat{\alpha}(n)$ is the minimizer of $\LL(\cdot)$ over $\NN_n$, we obtain
\begin{align*}
    \LL(\widehat{\alpha}(n))
    &\leq\inf_{\alpha\in\NN_n}\LL(\alpha)=\inf_{\alpha\in\NN_n}\|A\alpha-F\|^2_2\\
    &\lesssim\inf_{\alpha\in\NN_n}\left(\|A\alpha-A\alpha^*\|^2_2+\|A\alpha^*-F\|^2_2\right)\\
    &\lesssim\inf_{\alpha\in\NN_n}\|\alpha-\alpha^*\|^2_2,
\end{align*}
where we have used \eqref{LAS} and Proposition \ref{matrix_thm} once more. Then by the universal approximation property (Theorem \ref{UAT_main_thm}), $\inf_{\alpha\in\NN_n}\|\alpha-\alpha^*\|^2_2\rightarrow0$ as $n\rightarrow\infty$ and hence $\LL(\widehat{\alpha}(n))\rightarrow0$ as $n\rightarrow\infty$. Together with \eqref{mid_1}, this completes the proof. 
\end{proof}

\begin{remark}
In many cases, the explicit convergence rate for the approximation error can be computed. See, for instance, \cite{UAT_rate_1, UAT_rate_2} where the approximation of continuous functions with deep ReLU networks are discussed. However, the approximation rates for continuous functions deteriorate as the dimension increases. This can be tackled by assuming that the target function lies in a suitable smaller function
class which has low complexity compared to the functions with classical regularity. One function class of this type is the so-called Barron space introduced in the seminal work \cite{barron_1}. If the target function $\alpha^*$ belongs to the Barron space, then approximating it by two-layer networks has the explicit convergence rate $\mathcal{O}(n^{-\frac{1}{2}})$. See also \cite{barron_2, barron_3, rade_upper_2, barron_4, barron_5, rade_upper_3} for some variants of Barron spaces and their approximation properties by two-layer neural networks. In our case, if the external force $f(x;\om)$ is, for example, Gaussian, positive definite, radial, or smooth with respect to $\om\in\Omega$ (see \cite{barron_1, rade_upper_2} where the characterization of Barron functions is discussed), then $F(\om)$ is the Barron function, and hence our target function $\alpha^*(\om)=A^{-1}F(\om)$ becomes the Barron function.
\end{remark}

\subsection{Generalization error}
To handle the generalization error, we need to introduce the quantity so-called {\textit{Rademacher complexity}} \cite{Rade_3}. For more information on the Rademacher complexity and related topics, see \cite{Rade_1, Rade_2, Rade_4}.
\begin{definition}
For a family $\{X_i\}_{i=1}^M$ of i.i.d. random variables according to $\mathbb{P}_{\Omega}$, we define the Rademacher complexity of the function class $\mathcal{F}$ by
\[
    R_M(\mathcal{F})=\mathbb{E}_{\{X_i,\varepsilon_i\}^M_{i=1}}\bigg[\sup_{f\in\mathcal{F}}\bigg|\frac{1}{M}\sum^M_{i=1}\varepsilon_if(X_i)\bigg|\bigg],
\]
where $\varepsilon_i$'s are i.i.d. Bernoulli random variables, i.e., $\mathbb{P}(\varepsilon_i=1)=\mathbb{P}(\varepsilon_i=-1)=\frac{1}{2}$.
\end{definition}
Note from the above definition that the Rademacher complexity is the average of the maximum correlation between
the vector $(f(X_1),\cdots,f(X_M))$ and the random noise $(\varepsilon_1,\cdots,\varepsilon_M)$, where the maximum is
chosen over the function class $\mathcal{F}$. Intuitively, the Rademacher complexity of the function class $\mathcal{F}$ measures the ability of functions from $\mathcal{F}$ to fit random noise. If we can always find a function
which has a high correlation with a randomly selected noise vector, the corresponding function class can be interpreted as too large for statistical purposes. Conversely, if the Rademacher complexity decays as the sample size $M$ increases, then it is impossible to find a function that highly correlates in expectation with random noise.

We now make the precise connection between Rademacher complexity and the generalization error. In particular, we shall use the result that for any bounded function class $\mathcal{F}$, the condition $R_M({\mathcal{F}}) = o(1)$ implies the convergence of the generalization error, which is called the uniform law of large numbers. In the following theorem, the function class is assumed to be $b$-uniformly bounded, meaning that $\|f\|_{\infty}\leq b$ for all $f\in\mathcal{F}$.

\begin{theorem}\label{rade_thm}
{\rm{[Theorem 4.10 in \cite{Rade_2}]}} Let $\mathcal{F}$ be a $b$-uniformly bounded class of functions and $M\in\mathbb{N}$. Then for any small number $\delta>0$, we have
\[
\sup_{f\in\mathcal{F}}\bigg|\frac{1}{M}\sum^M_{i=1}f(X_i)-\mathbb{E}[f(X)]\bigg|\leq 2R_M(\mathcal{F})+\delta,
\]
with probability at least $1-\exp(-\frac{M\delta^2}{2b^2})$.
\end{theorem}

Let us now define the function class
\begin{equation}\label{fct_class}
    \mathcal{F}_n:=\{|A\alpha-F|^2:\alpha\in\NN_n\},
\end{equation}
where $A$ and $F$ are define in \eqref{vec_form}. Then by Proposition \ref{matrix_thm}, we have that
\[
    \|A\alpha-F\|_{L^{\infty}(\Omega)}\leq\|A\alpha\|_{L^{\infty}(\Omega)}+\|F\|_{L^{\infty}(\Omega)}\lesssim \|\alpha\|_{L^{\infty}(\Omega)}+\|F\|_{L^{\infty}(\Omega)}.
\]
Therefore, by \eqref{bddness_NN} and \eqref{f_ass}, we observe that for all $n\in\mathbb{N}$, $\mathcal{F}_n$ is $\overline{G}$-uniformly bounded for some positive constant $\overline{G}>0$. Then the following lemma is the direct consequence of Theorem \ref{rade_thm}.
\begin{lemma}\label{rade_main}
Let $\{\om_m\}^M_{m=1}$ be i.i.d. samples randomly chosen from the distribution $\mathbb{P}_{\Omega}$ in the definition of the empirical loss function \eqref{emp_loss}. Then for any small $\delta>0$, we have with probability at least $1-2 \exp(-\frac{M\delta^2}{32 \overline{G}^2})$
\begin{equation}\label{loss_diff}
    \sup_{\alpha\in\NN_n}\left|\LL^M(\alpha)-\LL(\alpha)\right|\leq2R_n(\mathcal{F}_n)+\frac{\delta}{2}.
\end{equation}
\end{lemma}


With Lemma \ref{rade_main}, we now establish the convergence concerning $M\rightarrow\infty$ for the generalization error, which is encapsulated in the following theorem. Here we assume that the Rademacher complexity of $\mathcal{F}_n$ converges to zero for all $n\in\mathbb{N}$, which is the standard assumption as discussed above.
\begin{theorem}\label{gen_conv_thm}
Assume that \eqref{f_ass} holds, and suppose that $\lim_{M\rightarrow\infty}R_M(\mathcal{F}_n)=0$ for all $n\in\mathbb{N}$. Then we have with probability $1$ over i.i.d. samples that
\[
    \lim_{n\rightarrow\infty}\lim_{M\rightarrow\infty}\|\widehat{\alpha}(n,M)-\widehat{\alpha}(n)\|^2_2=0.
\]
\end{theorem}
\begin{proof}
By Proposition \ref{matrix_thm} and the definition \eqref{for_2}, we have
\begin{equation}
    \begin{aligned}\label{gen_pr_mid}
        \|\widehat{\alpha}(n)-\widehat{\alpha}(n,M)\|^2_2
        &\lesssim \|A\widehat{\alpha}(n)-A\widehat{\alpha}(n,M)\|^2_2\lesssim \|A\widehat{\alpha}(n)-F\|^2_2 + \|A\widehat{\alpha}(n,M)-F\|^2_2\\
        &=\LL(\widehat{\alpha}(n))+\LL(\widehat{\alpha}(n,M))\leq 2\LL(\widehat{\alpha}(n,M))
    \end{aligned}
\end{equation}
We shall now apply Lemma \ref{rade_main} with $\delta=2M^{-\frac{1}{2}+\varepsilon}$ for $0<\varepsilon<\frac{1}{2}$. Then with probability $1-2\exp(-\frac{M^{2\varepsilon}}{8\overline{G}^2})$, we have by minimality of $\widehat{\alpha}(n,M)$,
\[
    \LL(\widehat{\alpha}(n,M))
    \leq \LL^M(\widehat{\alpha}(n,M))+2R_M(\mathcal{F}_n)+M^{-\frac{1}{2}+\varepsilon}\leq \LL^M(\widehat{\alpha}(n))+2R_M(\mathcal{F}_n)+M^{-\frac{1}{2}+\varepsilon}
\]
If we apply Lemma \ref{rade_main} once more, we obtain
\[
    \LL(\widehat{\alpha}(n,M))\leq \LL(\widehat{\alpha}(n))+4R_M(\mathcal{F}_n)+2M^{-\frac{1}{2}+\varepsilon}.
\]
Therefore, if we take $M\rightarrow\infty$ on \eqref{gen_pr_mid}, we have by the assumption on $R_M(\mathcal{F}_n)$ that
\[
    \lim_{M\rightarrow\infty}\|\widehat{\alpha}(n,M)-\widehat{\alpha}(n)\|^2_2\lesssim \LL(\widehat{\alpha}(n)).
\]
As before, from Theorem \ref{UAT_main_thm}, we observe by taking $n\rightarrow\infty$ that
\[
    \lim_{n\rightarrow\infty}\lim_{M\rightarrow\infty}\|\widehat{\alpha}(n,M)-\widehat{\alpha}(n)\|^2_2\lesssim \lim_{n\rightarrow\infty}\LL(\widehat{\alpha}(n))\lesssim \lim_{n\rightarrow\infty}\inf_{\alpha\in\NN_n}\|\alpha-\alpha^*\|^2_2\rightarrow 0,
\]
and the proof is completed.
\end{proof}

\begin{remark}
The assumption $\lim_{M\rightarrow\infty}R_M(\mathcal{F}_n)=0$ in Theorem \ref{gen_conv_thm} is common in the statistical learning theory and known for several function classes \cite{rade_upper_1, rade_upper_2}. Moreover, in many cases, it was shown that the Rademacher complexity converges to zero as $M$ goes to infinity with the convergence rate $\mathcal{O}(M^{-\frac{1}{2}})$, which is same as the rate for the classical Monte Carlo integration; see, for example, two-layer ReLU networks \cite{rade_upper_1},  spectral Barron setting with the softplus activation function  \cite{rade_upper_3}, general Barron class \cite{rade_upper_2} and recent result on the complexity estimate for general case \cite{rade_upper_4}.
\end{remark}

\section{Application to linear second-order elliptic PDEs}\label{PDE_app}
In this section, we will derive the convergence of approximate solutions for the specific PDE problems based on what we have discussed in the previous sections. In general, one of the most challenging tasks for neural network-based PDE techniques is to impose the boundary condition. One typical way to handle this problem is to add penalty terms to loss functions. However, adding the penalty term to the loss function make it difficult to analyze the algorithm and may deteriorate the convergence.

As mentioned earlier, one important advantage of using the ULGNet is that we can impose exact boundary values to the approximate solutions. More precisely, once the boundary condition is given (either Dirichlet or Neumann), $\{a_k,b_k\}_{k=1}^{N-1}$ in \eqref{basis} are uniquely determined and hence, the basis function \eqref{basis} is fixed with the exact boundary conditions. Then we compute the stiffness matrix $S$ and the mass matrix $M$ using \eqref{vec_form}, and hence the total matrix $A$ is obtained accordingly. Then we can apply the abstract analysis studied in the previous section. One interesting thing is that the analysis in Section \ref{aca} can cover most of the self-adjoint PDEs. In the following subsection, we will deal with both the Dirichlet case and the Neumann case with detailed computation, to highlight the flexibility of the proposed method. As mentioned earlier, for the sake of simplicity we shall consider the self-adjoint equation
\begin{equation}\label{main_anal_eq}
    -u''(x;\om)+\nu u(x;\om) = f(x;\om)\quad{\rm{for}}\,\,x\in I:=(-1,1),\,\,\om\in\Omega,
\end{equation}
with the homogeneous boundary conditions
\begin{equation}\label{main_bc}
    u(-1;\om)=u(1;\om)=0\quad{\rm{or}}\quad u'(-1;\om)=u'(1;\om)=0\quad\forall\om\in\Omega.
\end{equation}
Note, however, that our analysis can be extended to the case of any self-adjoint PDEs and non-homogeneous boundary conditions in a straightforward manner. After that, we will also discuss the way to extend the analysis to the higher-dimensional case. Throughout the section, we will use the same notations, assumptions, and mathematical settings introduced in Section \ref{prelim}.

\subsection{Main theorem and convergence of approximate solutions}
For the input of neural networks, we shall assume that the forcing functions are generated with respect to the parameter $\om$ uniformly distributed over the $d$-dimensional unit cube, i.e. the external force is regarded as a bivariate function $f(x;\om)$ where the spatial variable $x$ is defined over $I$ and the parametric variable $\om$ is uniformly distributed over $[a,b]^d$ for $a,b \in \mathbb{R}^+$. Note that as the dimension $d$ increases, our method can cover a wider range of external force functions. As mentioned earlier, we assume that \eqref{f_ass} holds.

We begin with the following lemma, concerning the values of $a_k$ and $b_k$, which is a direct consequence of Lemma 4.1 in \cite{spectral_method}.

\begin{lemma}\label{coeff_val}
Let $L_k$ be the $k$-th Legendre polynomial. For each $k\geq0$, there is a unique set of $\{a_k,b_k\}^{N-1}_{k=1}$ such that $\phi_k(x)=L_k(x)+a_kL_{k+1}(x)+b_kL_{k+2}(x)$ satisfies the given boundary condition (either Dirichlet or Neumann). In particular,
\begin{itemize}
    \item homogeneous Dirichlet boundary condition $(\phi_k(-1)=\phi_k(1)=0$ for all $k\geq1)$:
    \begin{equation}\label{dd_coef}
        a_k=0\quad{\rm{and}}\quad b_k=-1\quad\forall k\geq1.    
    \end{equation}
    \item homogeneous Neumann boundary condition $(\phi'_k(-1)=\phi'_k(1)=0$ for all $k\geq1)$:
    \begin{equation}\label{nn_coef}
        a_k=0\quad{\rm{and}}\quad b_k=\frac{-k(k+1)}{(k+2)(k+3)}\quad\forall k\geq1.
    \end{equation}
\end{itemize}
\end{lemma}

We will also use the following lemma, regarding the structure of the stiffness and the mass matrices, again quoted from \cite{spectral_method}.

\begin{lemma}\label{matrix_struc}
The stiffness matrix $S$ defined in \eqref{vec_form} is a diagonal matrix with the diagonal entries
\begin{equation}\label{stiff_mat}
    S_{kk}=-(4k+6)b_k\quad\forall k\geq1.
\end{equation}
Furthermore, the mass matrix $M$ is a symmetric penta-diagonal matrix with the following non-zero elements:
\begin{equation}\label{mass_mat}
    M_{jk}=M_{kj}=
    \begin{cases}
    \frac{2}{2k+1}+\frac{2a_k^2}{2k+3}+\frac{2b^2_k}{2k+5},\quad&j=k,\\
    \frac{2a_k}{2k+3}+\frac{2a_{k+1}b_k}{2k+5},\quad&j=k+1,\\
    \frac{2b_k}{2k+5},\quad&j=k+2.
    \end{cases}
\end{equation}
\end{lemma}
Note from Lemma \ref{coeff_val} and Lemma \ref{matrix_struc} that all entries of the stiffness and mass matrices are finite, since the basis functions we used are continuous over the compact interval  $[-1,1]$, and hence the integrals in \eqref{vec_form} are well-defined.

Let us write the stiffness and the mass matrices as $S_{\rm{Dir}}$, $M_{\rm{Dir}}$ for the Dirichlet case computed by \eqref{dd_coef}, \eqref{stiff_mat} and \eqref{mass_mat}, and define $A_{\rm{Dir}}=S_{\rm{Dir}}+\nu M_{\rm{Dir}}$. The matrices for the Neumann case $S_{\rm{Neu}}$, $M_{\rm{Neu}}$ and $A_{\rm{Neu}}$ are similarly obtained by using \eqref{nn_coef}, \eqref{stiff_mat} and \eqref{mass_mat}. From now on, $A_{\rm{BC}}$ means either $A_{\rm{Dir}}$ or $A_{\rm{Neu}}$ and the matrix denoted by $A$ in the previous sections is now considered to be $A_{\rm{BC}}$. By Lemma \ref{matrix_struc}, $A_{\rm{BC}}$ is symmetric. Furthermore for any non-zero $y=(y_1,\cdots,y_{N-1})\in\R^{N-1}$,
\begin{equation}\label{pos_def}
    y\cdot A_{\rm{BC}}y
    =\sum^{N-1}_{i=1}\sum^{N-1}_{j=1}y_iy_j\int_I\left(\phi'_i(x)\phi'_j(x)+\nu\phi_i(x)\phi_j(x)\right)\dx\geq\int_I\sum^{N-1}_{i=1}y_i^2|\phi'_i(x)|^2\dx=\sum^{N-1}_{i=1}y_i^2S_{ii}>0,
\end{equation}
which means that $A_{\rm{BC}}$ is positive-definite, and hence is non-singular. Therefore, we can apply the analysis conducted in Section \ref{aca}. We now consider that the loss functions $\LL$ and $\LL^M$ (\eqref{obs_J} and \eqref{obs_JM} respectively) where the matrix $A$ is now replaced by $A_{\rm{BC}}$, and the corresponding coefficients defined through \eqref{for_1}, \eqref{for_2} and \eqref{for_4}. Then Theorem \ref{approx_main_thm} and Theorem \ref{gen_conv_thm} together with the triangular inequality give us that
\begin{equation}\label{conv_1}
    \lim_{n\rightarrow\infty}
    \lim_{M\rightarrow\infty}\|\alpha^*-\widehat{\alpha}(n,M)\|_{L^2(\Omega)}=0.
\end{equation}

Now we shall state and prove the main theorem of this paper. Recall that the definition of approximate solutions \eqref{for_0}, \eqref{for_3} and \eqref{for_5}.
\begin{theorem}\label{main_thm_whole}
Assume that \eqref{f_ass} holds, and suppose that for any $n\in\mathbb{N}$, $R_M(\widetilde{{\mathcal{F}}}_n)\rightarrow0$ as $M\rightarrow \infty$, where $\widetilde{{\mathcal{F}}}_n:=\{|A_{\rm{BC}}\alpha-F|^2:\alpha\in\NN_n\}$. Then for fixed $N\in\mathbb{N}$, we have with probability $1$ over i.i.d. samples that
\begin{equation}\label{main_conv_whole}
    \lim_{n\rightarrow \infty}\lim_{M\rightarrow \infty}\|u_N-u_{N,n,M}\|_{L^2(\Omega;H^1(I))}=0.
\end{equation}
\end{theorem}
\begin{proof}
Let us first prove the convergence with respect to the $H^1$-seminorm. By the orthogonality of $\{\phi'_k\}_{k=1}^{N-1}$ (Lemma \ref{coeff_val}, Lemma \ref{matrix_struc}), and Theorem \ref{approx_main_thm}, Theorem \ref{gen_conv_thm}, we note that
\begin{align*}
    \int_{\Omega}\int_{I}|u'_N-u'_{N,n,M}|^2\dx\dom
    &=\int_{\Omega}\int_I\bigg|\sum^{N-1}_{k=1}\left(\alpha^*_k-\widehat{\alpha}(n,M)_k\right)\phi'_k\bigg|^2\dx\dom\\
    &=\int_{\Omega}\sum^{N-1}_{k=1}|\alpha^*_k-\widehat{\alpha}(n,M)_k|^2S_{kk}\dom\\
    &\lesssim\|\alpha^*-\widehat{\alpha}(n,M)\|^2_{L^2(\Omega)}\rightarrow0\quad{\rm{as}}\,\,M,n\rightarrow\infty.
\end{align*}
For the Dirichlet case, the full $H^1$-convergence easily follows by the use of Poincar\'e's inequality. Here, however, to cover the Neumann case as well, we present a more general argument. From the structure of the mass matrix (Lemma \ref{coeff_val}, Lemma \ref{matrix_struc}) and Young's inequality, we obtain that
\begin{align*}
    \int_{\Omega}\int_{I}|u_N-u_{N,n,M}|^2\dx\dom
    &=\int_{\Omega}\int_I\bigg|\sum^{N-1}_{k=1}\left(\alpha^*_k-\widehat{\alpha}(n,M)_k\right)\phi_k\bigg|^2\dx\dom\\
    &=\int_{\Omega}\sum^{N-1}_{k=1}|\alpha^*_k-\widehat{\alpha}(n,M)_k|^2M_{kk}\dom\\
    &\,\,\,\,\,\,\,+2\int_{\Omega}\sum^{N-3}_{j=1}|\alpha^*_j-\widehat{\alpha}(n,M)_j||\alpha^*_{j+2}-\widehat{\alpha}(n,M)_{j+2}|M_{j\,j+2}\dom\\
    &\lesssim\int_{\Omega}\sum^{N-1}_{k=1}|\alpha^*_k-\widehat{\alpha}(n,M)_k|^2\dom+\int_{\Omega}\sum^{N-3}_{j=1}|\alpha^*_j-\widehat{\alpha}(n,M)_j|^2\dom\\
    &\,\,\,\,\,\,\,+\int_{\Omega}\sum^{N-3}_{j=1}|\alpha^*_{j+2}-\widehat{\alpha}(n,M)_{j+2}|^2\dom\\
    &\lesssim\|\alpha^*-\widehat{\alpha}(n,M)\|^2_{L^2(\Omega)}\rightarrow0\quad{\rm{as}}\,\,M,n\rightarrow\infty,
\end{align*}
where we have used once more Theorem \ref{approx_main_thm} and Theorem \eqref{gen_conv_thm}.
\end{proof}


\subsection{Analysis in higher dimensions}\label{sec:multi_dim}
In this section, we shall discuss the way to extend the ULGNet scheme and its analysis to the multi-dimensional case for the spatial variable. Here we illustrate the two-dimensional case with homogeneous Dirichlet boundary condition, and the study for general-dimensional cases follows in a straightforward manner.

Let us consider the following two-dimensional problem:
\begin{align}
    -\Delta u + \nu u&=f\quad{\rm{in}}\,\,I^2=[-1,1]^2,\label{2d_eq_1}\\
    u&=0\quad{\rm{on}}\,\,\partial I^2.\label{2d_eq_2}
\end{align}
Here $f$ is a two-dimensional function, again parametrized by $\om\in\Omega$. The weak solution of the equation \eqref{2d_eq_1}-\eqref{2d_eq_2} is the function $u\in H^1_0$ satisfying the weak formulation
\begin{equation}\label{2d_WF}
    \int_{I^2}\left(\nabla u(x,y)\cdot\nabla v(x,y)+\nu u(x,y)v(x,y)\right)\dx\dy=\int_{I^2}f(x,y)v(x,y)\dx\dy,\quad\forall v\in H^1_0.
\end{equation}
For the Legendre--Galerkin approximation of \eqref{2d_eq_1}-\eqref{2d_eq_2}, we define the two-dimensional basis function by
\begin{equation}\label{2d_basis}
    \phi_{i,j}(x,y)=\phi_i(x)\phi_j(y)\quad i,j\in\{1,\cdots,N-1\},
\end{equation}
where $\phi_i(x)$ and $\phi_j(y)$ are one-dimensional basis functions defined in \eqref{basis}. Note that $(a_k,b_k)$'s are determined with the same values computed in \eqref{dd_coef}, since $\phi_{i,j}(x,y)$ is reduced to the single-variable basis function if we fix one of $x$ or $y$, which needs to satisfy the one-dimensional homogeneous Dirichlet boundary condition as before.

We shall rewrite $\phi_{i,j}(x,y)$, $i,j\in\{1,\cdots,N-1\}$ as $\overline{\phi}_k(x,y)$, $k\in\{1,\cdots,(N-1)^2\}$, by allocating the multi indices $\{(1,1),\cdots,(1,N-1),(2,1),\cdots,(2,N-1),\cdots,(N-1,1),\cdots,(N-1,N-1)\}$ to $\{1,\cdots,(N-1)^2\}$. As before, for $i$, $j\in\{1,\cdots,(N-1)^2\}$, we define the stiffness matrix, the mass matrix and the load vector by
\begin{equation}\label{2d_mats}
    \begin{aligned}
        \overline{S}&=(\overline{S}_{ij})&&{\rm{where}}\,\,\overline{S}_{ij}:=\int_{I^2}\overline{\phi}'_i(x,y)\overline{\phi}'_j(x,y)\dx\dy,\\
        \overline{M}&=(\overline{M}_{ij})&&{\rm{where}}\,\,\overline{M}_{ij}:=\int_{I^2}\overline{\phi}_i(x,y)\overline{\phi}_j(x,y)\dx\dy,\\
        \overline{F}&=(\overline{F}_{j})&&{\rm{where}}\,\,\overline{F}_{j}:=\int_{I^2}f(x,y)\overline{\phi}'_j(x,y)\dx\dy,
    \end{aligned}
\end{equation}
and we let $\overline{A}=\overline{S}+\nu\overline{M}$. Then we define the loss functions
\begin{equation}\label{2d_loss_fct}
\overline{\LL}(\alpha)=\|\overline{A}\alpha(\om)-\overline{F}(\om)\|^2_{L^2(\Omega)}\quad{\rm{and}}\quad\overline{\LL}^M(\alpha)=\frac{|\Omega|}{M}\sum^M_{m=1}|\overline{A}\alpha(\om_m)-\overline{F}(\om_m)|^2, 
\end{equation}
and as before, we further define the Legendre--Galerkin approximation
\begin{equation}
    \overline{\alpha}^*=\argmin_{\alpha\in C(\Omega,\R^{(N-1)^2})}\overline{\LL}(\alpha),\quad\overline{u}_N=\sum^{(N-1)^2}_{k=1}\overline{\alpha}^*_k\overline{\phi}_k(x),
\end{equation}
the continuous residual minimization
\begin{equation}
    \overline{{\alpha}}(n)=\argmin_{\alpha\in \overline{N}_n}\overline{\LL}(\alpha),\quad\overline{u}_{N,n}=\sum^{(N-1)^2}_{k=1}\overline{{\alpha}}(n)_k\overline{\phi}_k(x),
\end{equation}
and the discrete residual minimization
\begin{equation}
  \overline{{\alpha}}(n,M)=\argmin_{\alpha\in \overline{N}_n}\overline{\LL}^M(\alpha),\quad\overline{u}_{N,n,M}=\sum^{(N-1)^2}_{k=1}\overline{{\alpha}}(n,M)_k\overline{\phi}_k(x),\\  
\end{equation}
where $\overline{N}_n$ consists of neural networks with $(N-1)^2$-dimensional output.

Next, by the definition, $\overline{A}$ is symmetric. Moreover, by the same argument used in \eqref{pos_def}, $\overline{A}$ is positive-definite, and hence non-singular. Therefore, we can apply Proposition \ref{matrix_struc}, from which we can deduce the convergence
\[
\lim_{n\rightarrow\infty}\lim_{M\rightarrow\infty}\|\overline{\alpha}^*-\overline{\alpha}(n,M)\|^2_{L^2(\Omega)}=0
\]
with the same arguments proposed in Theorem \ref{approx_main_thm} and Theorem \ref{gen_conv_thm} with $A=\overline{A}$.
Again, by the same argument presented in Theorem \ref{main_thm_whole} with $A_{\rm{BC}}$ replaced by $\overline{A}$, we conclude that
\begin{equation}\label{2d_main_conv_whole}
    \lim_{n\rightarrow \infty}\lim_{M\rightarrow \infty}\|\overline{u}_N-\overline{u}_{N,n,M}\|_{L^2(\Omega;H^1(I))}=0.
\end{equation}

\section{Numerical experiments}\label{num_exp}
In this section, we shall demonstrate some numerical results to verify our theoretical findings presented in the previous sections. We consider the equation \eqref{main_eq} with $\nu=1$, and impose the homogeneous Dirichlet boundary condition on the interval $[-1,1]$:
\begin{align}
    -u''+u&=f\quad{\rm{in}}\,\,(-1,1),\label{exp_eq_1}\\
    u(-1)&=u(1)=0.\label{exp_eq_2}
\end{align}
For the external forcing term $f$, we randomly generate $M$ samples of the training data of the ULGNet which is given by 
\begin{equation}\label{f_exp}
    f(x;\om_m)=\om^1_m\sin(2\pi\om^2_mx)+\om^3_m\cos(2\pi\om^4_mx),\quad1\leq m\leq M,
\end{equation}
where the random sample $\om_m$ is selected from the uniform distribution
\[{\rm{i.i.d.}}\quad\om_m=(\om^1_m,\om^2_m,\om^3_m,\om^4_m)\sim[0,1]^4=:\Omega,\quad1\leq m\leq M.\]

For the spectral element approximation of the equations \eqref{exp_eq_1}-\eqref{exp_eq_2}, as presented in Lemma \ref{coeff_val} and Lemma \ref{matrix_struc}, we will utilize the basis function \eqref{basis} and hence the matrices \eqref{stiff_mat}, \eqref{mass_mat} with
\[a_k=0\quad{\rm{and}}\quad b_k=1\quad\forall k\geq1.\]
In our experiments, we set
N = 32 and use the Legendre-Gauss-Lobatto nodal points for numerical integrations. Then we solve the algebraic equation \eqref{LAS} to compute $u_N$ which will be regarded as true solutions in the present experiment. For the approximate solutions using the ULGNet, we minimize the empirical loss function defined by \eqref{emp_loss} to compute the coefficients $\widehat{\alpha}(n,M)$ and construct the predicted solution $U_{N,n,M}$ as \eqref{for_5}. We then measure the distance
\[\|u_N-u_{N,n,M}\|_{L^2(\Omega;L^2(I))}\]
by using the Monte--Carlo simulation with respect to $\Omega$ with $1000$ random samples. In this setting, we shall conduct the numerical experiments with two different scenarios. One is to grow the neural network architecture with fixed $M\in \mathbb{N}$, and the other one is to fix the neural network architecture and to increase the number of training samples $M\in\mathbb{N}$.

More precisely, for the first experiment, we consider the two-layer neural networks with $n\in\mathbb{N}$ neurons in the hidden layer, and use the hyperbolic tangent activation function which satisfies the boundedness condition assumed in Theorem \ref{UAT_main_thm}. For the training, We fix the number of training data $M=10000$ and adopt the limited-memory Broyden--Fletcher--Goldfarb--Shanno (L-BFGS) optimization algorithm as in \cite{optim}, with $20000$ training epochs. Then for each $n\in\mathbb{N}$ varying from $20$ to $300$, we train the model as described above and compute the relative $L^2$-error between $u_N$ and $U_{N,n,M}$ evaluated over $1000$ randomly generated functions of the form \eqref{f_exp} that were not used in the training.

On the other hand, for the second experiment, we adopt a fully-connected neural networks with five hidden layers with $128$ neurons for each layer, again equipped with the hyperbolic tangent activation function. We fix this neural network architecture and train the models with the same method described above with the different number of training data $M\in\mathbb{N}$ for each model. In both experiments, all the computations were carried on by Python software on the GPU cluster in the Department of Mathematics, Sugnkyunkwan University. The results of the experiments are presented in Figure \ref{ULGNet_num_exp}.

\begin{figure}
\centering
\subfloat[Relative $L^2$-error as $n$ increases. ]{{\includegraphics[width=0.45\textwidth ]{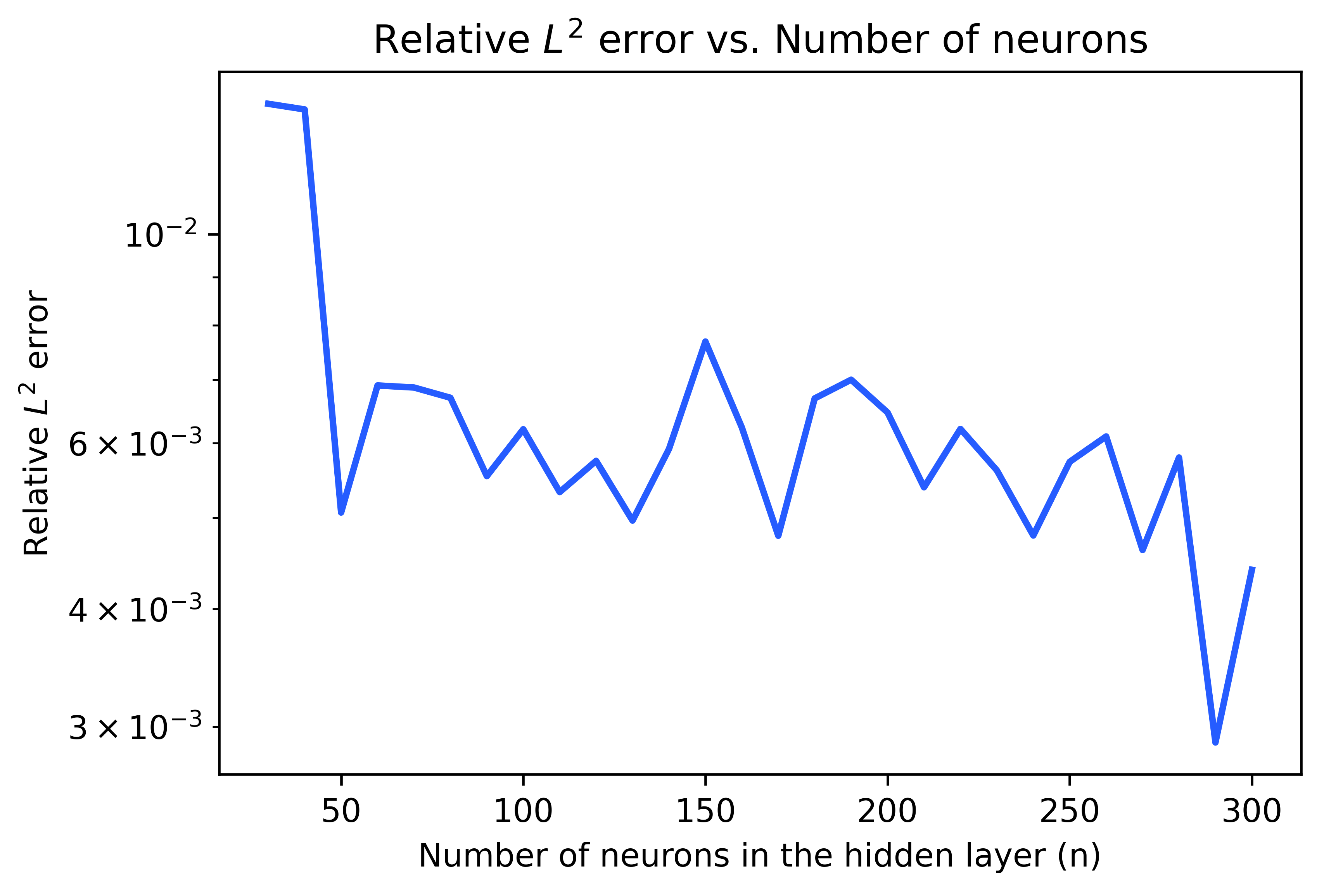} }}%
\subfloat[Relative $L^2$-error as $M$ increases.]{{\includegraphics[width=0.435\textwidth ]{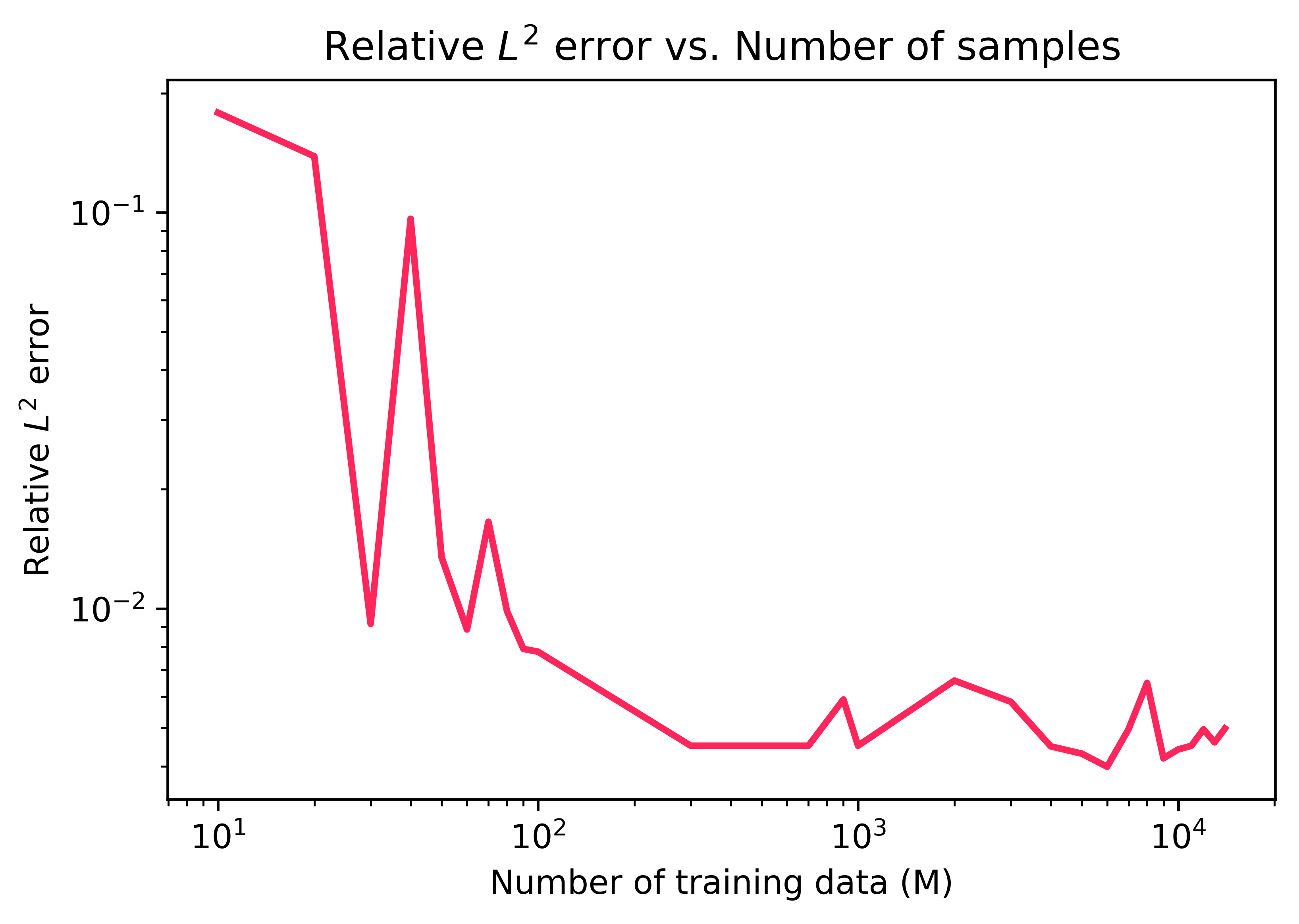} }}%
\caption{Relative $L^2$-error evaluated over $1000$ randomly generated samples as $n$, $M$ increase.}%
\label{ULGNet_num_exp}%
\end{figure}
As shown in Figure \ref{ULGNet_num_exp}, we can say that the relative $L^2$ error tends to decrease as $n$ and $M$ increase. Since there are a number of factors in training machine learning algorithms which may 
hold us back from obtaining precise values (e.g. there is randomness for each realization of $\om_m$ and optimization errors may occur), the trend of errors in the figure is not expressed in a deterministic manner as the classical numerical methods for PDEs. However, it can be clearly confirmed that from the figure that the error tends to decrease in both cases, which supports what we theoretically proved in previous sections.

\section{Conclusion}\label{con_rmk}
In this paper, we have established the convergence of approximate solutions obtained by the ULGNet algorithm to the weak solution of the linear, second-order elliptic PDEs. The idea of ULGNet is to represent solutions as a linear combination of the Legendre--Galerkin basis functions and generate the coefficients through the deep neural network whose input is an external forcing term, where the loss function is defined via the variational residual minimization. The key observation for the proof is to convert the variational minimization problem to the strong residual minimization solving matrix equation, which is similar to the framework of classical numerical schemes. Then we perform the analysis for the derived matrix, enabled by the choice of our Legendre-Galerkin basis functions. As a consequence, together with the universal approximation property of neural networks and the convergence of Rademacher complexity, we showed the strong convergence of the minmizers to the solutions of the given PDEs.

We expect that the numerical scheme and the analysis studied in the present paper have considerable potential to be further investigated in various
directions. As the well-known basis functions are used for the scheme, from the analysis point of view,
it is more likely for us to apply various theories from classical numerical analysis. One interesting future research direction is to use different basis functions which might lead us to achieve better computational cost and numerical tractability. On the other hand, it would be of particular interest to extend the result to nonlinear equations with the iterative methods, or the time-stepping analysis for the parabolic problem, which will be addressed in the forthcoming papers.


\bibliography{references}
\bibliographystyle{abbrv}


\end{document}